\documentclass{article}
\usepackage{graphicx} 
\usepackage{amsfonts,amsmath,amssymb,amsthm}
\usepackage{xcolor}

\newtheorem{proposition}{Proposition}

\newtheorem{example}{Example}[section]

\def\jstar{J^*}

\def\a{\alpha}
\def\m{\mu}

\title{An Error Bound for Aggregation in Approximate Dynamic Programming\footnote{
This work was carried out at the Fulton School of Computing and Augmented Intelligence, Arizona State University, Tempe, AZ.}}
\author{Yuchao Li \  and \ Dimitri Bertsekas}
\date{May 2026}

\begin{document}

\maketitle
\begin{abstract}
We consider a general aggregation framework for discounted finite-state infinite horizon dynamic programming (DP) problems. It defines an aggregate problem whose optimal cost function can be obtained off-line by exact DP and then used as a terminal cost  approximation for an on-line reinforcement learning (RL) scheme. We derive a bound on the error between the optimal cost functions of the aggregate problem and the original problem. This bound was first derived by Tsitsiklis and van Roy \cite{tsitsiklis1996} for the special case of hard aggregation. Our bound is similar but applies far more broadly, including to soft aggregation and feature-based aggregation schemes.
\end{abstract}

\section{The Aggregation Framework}\label{sec:intro}

We will focus on the standard discounted infinite horizon Markovian decision problem with the $n$ states  
$1,\ldots,n$. States and successor states will be denoted by $i$ and $j$.
State transitions $(i,j)$ under control $u$ occur at discrete times according to transition probabilities $p_{ij}(u)$, and generate a cost $\alpha^k g(i,u,j)$ at stage $k$, where $\alpha\in (0,1)$ is the discount factor. 

We consider deterministic stationary policies $\m$ such that for each $i$, $\m(i)$ is a control that belongs to a finite constraint set $U(i)$. We denote by $J_\m(i)$ the total discounted expected cost  of $\m$ over an infinite number of stages starting from state $i$, by $\jstar(i)$ the minimal value of $J_\m(i)$ over all $\m$, and by $J_\m$ and $\jstar $ the $n$-dimensional vectors that have components $J_\m(i)$ and $\jstar(i)$, $i=1,\ldots,n$, respectively.

We consider the general aggregation framework first described in the 2012 DP textbook by the second author \cite{bertsekas2012dynamic} (see Fig.~\ref{fig:aggregation}), following earlier, more specialized, frameworks. In particular, we introduce a finite subset ${\cal A}$ of aggregate states, which we denote  by symbols such as $x$ and $y$, together with two types of probability distributions as follows:
\begin{itemize}
\item[(a)] For each aggregate state $x\in {\cal A}$, a probability distribution over $\{1,\ldots,n\}$, denoted by 
$$\{d_{xi}\,|\, i=1,\ldots,n\},$$ and referred to as the {\it disaggregation probabilities of $x$\/}.	
\item[(b)]For each original system state $j\in \{1,\ldots,n\}$, a probability distribution over ${\cal A}$, denoted by 
$$\{\phi_{jy}\,|\, y\in {\cal A}\},$$
and referred to as the {\it aggregation probabilities of $j$\/}. 
\end{itemize}

The aggregation and disaggregation probabilities  may be viewed as the parameters of our aggregation architecture. Together with the set of aggregate states ${\cal A}$, they specify a DP problem, called the {\emph{aggregate problem}. In this DP problem, the corresponding dynamic system involves two copies of the original state space $\{1,\ldots,n\}$ as well as the aggregate states, with transitions and associated costs defined as shown in Fig.~\ref{fig:aggregation}. 
In particular, a single transition in the aggregate problem, starting at an aggregate state $x$, involves three stages and ends up at another aggregate state $y$ as follows:
\begin{itemize}
\item[(i)]From aggregate state $x$, we generate a cost-free transition to an original system state $i$ according to  $d_{xi}$.
\item[(ii)]We generate a transition between original system states $i$ and $j$ according to $p_{ij}(u)$, with cost $g(i,u,j)$. 	
\item[(iii)]From original system state $j$, we generate a cost-free  transition to an aggregate
state $y$ according to  $\phi_{jy}$.
\end{itemize}

\begin{figure}
    \centering
    \includegraphics[width=0.95\linewidth]{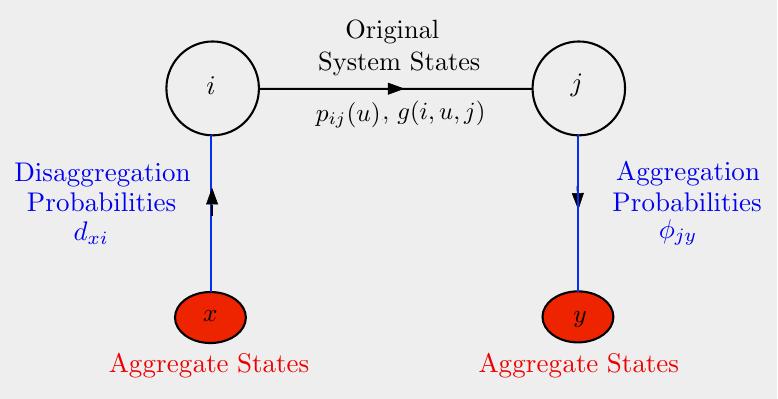}
    \caption{Illustration of our aggregation framework. It describes the aggregate problem, and the corresponding transition mechanism and costs per stage. The state space of the aggregate problem is ${\cal A}\times\{1,\ldots,n\}\times \{1,\ldots,n\}$, with transitions as shown.}
    \label{fig:aggregation}
\end{figure}

 We introduce the optimal cost vectors of the aggregate problem, $\tilde J_0=\{J_0(i)\,|\, i= 1,\ldots,n\}$, $\tilde J_1=\{J_1(j)\,|\, j=1,\ldots,n\}$, and $r^*=\{r^*_x\,|\, x\in{\cal A}\}$, which give the infinite horizon optimal cost starting from each of the states of the three portions of the state space:
\begin{itemize}
    \item[] $r^*_x$ is the optimal cost-to-go from aggregate state $x$.
    \item[] $\tilde J_0(i)$ is the optimal cost-to-go from original system state $i$ that has just been generated from an  aggregate state (left side of Fig.~\ref{fig:aggregation}). 
    \item[] $\tilde J_1(j)$ is the optimal cost-to-go from original system state $j$ that has just been generated from an original system state (right side of Fig.~\ref{fig:aggregation}).
\end{itemize}
Note that because of the intermediate transitions to aggregate states, the vectors $\tilde J_0$ and $\tilde J_1$ are different.

These three vectors satisfy the following three Bellman equations:
\begin{gather}
    r^*_x=\sum_{i=1}^nd_{x i}\tilde J_0(i),\qquad x\in{\cal A},\label{eq:belman_1}\\
\tilde J_0(i)=\min_{u\in U(i)}\sum_{j=1}^n p_{ij}(u)\big(g(i,u,j)+\a \tilde J_1(j)\big),\qquad i=1,\ldots,n,\label{eq:belman_2}\\
\tilde J_1(j)=\sum_{y\in{\cal A}}\phi_{jy}r^*_y,\qquad j=1,\ldots,n.\label{eq:belman_3}
\end{gather}
Our objective is to solve for the optimal cost vector $r^*$ of the aggregate states in order to obtain approximations to the optimal costs $J^*(i)$ for the original problem through the interpolation formula
\begin{equation}
    \label{eq:j_tilde}
    \tilde J(j)=\sum_{y\in{\cal A}}\phi_{jy}r^*_y,\qquad j=1,\ldots,n.
\end{equation}
There are several DP methods to obtain $r^*$, including some that are simulation-based; see \cite{bertsekas2012dynamic}, Section 6.5, for an overview.

The preceding aggregation framework is described in detail in the book \cite{bertsekas2012dynamic}, and generalizes earlier aggregation frameworks from several works; see Singh, Jaakkola,
and Jordan \cite{singh1994reinforcement}, Gordon \cite{gordon1995stable}, and Tsitsiklis and Van Roy \cite{tsitsiklis1996}, \cite{van1995}, and the neuro-dynamic programming book \cite{bertsekas1996neuro}, Sections 3.1.2 and 6.7. 
There are several special cases that have received attention in the literature:

\begin{itemize}
\item[(a)] \emph{Hard aggregation}: Here all aggregation probabilities $\phi_{jy}$ are either 0 or 1. It  follows that the sets 
$$S_y=\{j\,|\, \phi_{jy}=1\},\qquad y\in{\cal A},$$
called the \emph{footprints} of the aggregate states $y$, form a partition of the original state space $\{1,\ldots,n\}$. Moreover, based on  Eq.\ \eqref{eq:j_tilde}, the cost approximation $\tilde J$ is piecewise constant: it is constant over each footprint set $S_y$, with value $r^*_y$.
\item[(b)] \emph{Soft aggregation}: This is an extension of hard aggregation, where there is a ``soft" boundary between the sets of the state space partition, i.e., the footprint sets overlap partially. The aggregation probabilities are chosen to be positive for the states of overlap, so that the cost approximation $\tilde J$ is piecewise constant, except along the states of footprint overlap, where $\tilde J$ changes ``smoothly." 
\item[(c)] \emph{Aggregation with representative states}: This is a common discretization or ``coarse grid" scheme, whereby  we choose a subset of “representative” original system states, and we
associate each one of them with an aggregate state. In particular, each aggregate state
$x$ is associated with a unique representative state $i_x$, and the disaggregation
probabilities are
$d_{x i} =1$ if $i= i_x$ and 0 otherwise. 
\item[(d)] \emph{Aggregation with representative features}: Here the aggregate states are characterized by nonempty subsets of original system states, which, however, may not form a  partition of the original state space. In an important
example of this scheme, we choose a collection of distinct “representative”
feature vectors, and we associate each one of them with an aggregate state
consisting of the subset of original system states that share the corresponding
feature value (see \cite{bertsekas2012dynamic}, Section 6.5, or \cite{bertsekas2019reinforcement}, Section 6.2). The paper \cite{bertsekas2018feature} provides an overview of feature-based aggregation, and discusses ways to combine the methodology with the use of deep neural networks.
\end{itemize}

An important question is to estimate the approximation error 
$$\max_{i=1,\ldots,n}|J^*(i)-\tilde J(i)|.$$ 
This question has been addressed by Tsitsiklis and van Roy \cite{tsitsiklis1996} for the special case of hard aggregation, under the  condition that for any pair $(x,i)$, $d_{xi}>0$ implies $\phi_{ix}=1$, i.e., if each aggregate state $x$ disaggregates exclusively within the corresponding footprint set $S_x=\{j\,|\, \phi_{jx}=1\}$. 

The purpose of this paper is to provide an extension of this bound, which holds beyond the case of hard aggregation, subject  to the condition
\begin{equation}
    \label{eq:critical_condition}
    d_{xi}>0\quad \hbox{implies} \quad \phi_{ix}>0,\qquad \hbox{for all }x\in \mathcal{A}\hbox{ and }i=1,\dots,n.
\end{equation}
This condition is natural in all the aggregation schemes described earlier, including soft aggregation, and we will show by example that the condition is essential for any kind of meaningful bound on $\max_{i=1,\ldots,n}|J^*(i)-\tilde J(i)|$ to hold.

\section{The Error Bound}\label{sec:bound}

By combining the three Bellman equations \eqref{eq:belman_1}-\eqref{eq:belman_3}, we see that $r^*$ satisfies
\begin{equation}
    \label{eq:bellman_comp}
    r^*_x=\sum_{i=1}^nd_{x i}\min_{u\in U(i)}\sum_{j=1}^n p_{ij}(u)\Bigg(g(i,u,j)+\a \sum_{y\in{\cal A}}\phi_{jy}\,r^*_{y}\Bigg),\qquad x\in{\cal A},
\end{equation}
or equivalently $r^*=Hr^*$, where  $H$ is the operator that maps the vector $r$ to the vector $Hr$ with components
\begin{equation}
    \label{eq:h_op}
    (Hr)(x)=\sum_{i=1}^nd_{xi}\min_{u\in U(i)}\sum_{j=1}^n p_{ij}(u)\Bigg(g(i,u,j)+\a \sum_{y\in{\cal A}}\phi_{jy}\,r_{y}\Bigg),\quad x\in{\cal A}.
\end{equation}

It can be seen that $H$ is monotone, in the sense that 
\begin{equation}
    \label{eq:monotone}
    Hr\ge Hr'\qquad \hbox{for all }r,r'\hbox{ such that } r\ge r'.
    \end{equation}
Moreover, it can be shown that \emph{$H$ is a contraction mapping with respect to the maximum norm}, and thus the composite Bellman equation \eqref{eq:bellman_comp} has $r^*$ as its unique solution; see \cite[Section~6.5]{bertsekas2012dynamic} or \cite[Section~6.2]{bertsekas2019reinforcement}. 

\begin{proposition}[Error Bound on Aggregation]
Let 
the condition \eqref{eq:critical_condition} hold.
Then we have 
$$|J^*(i)-\tilde J(i)|\leq \frac{\epsilon}{1-\alpha},\qquad i=1,2,\dots,n,$$
where 
$$\tilde J(i)=\sum_{x\in \mathcal{A}}\phi_{ix}r^*_x,$$
[cf.\ Eq.~\eqref{eq:j_tilde}], and 
$$\epsilon=\max_{x\in \mathcal{A}}\max_{\{i,j\,|\,\phi_{ix}>0,\,\phi_{jx}>0\}}|J^*(i)-J^*(j)|.$$
\end{proposition}

\begin{proof}  
    Our line of proof follows the proof of Prop.\ 6.8 in the book by Bertsekas and Tsitsiklis \cite{bertsekas1996neuro}. Consider the operator $H$ defined by Eq.~\eqref{eq:h_op}, and the vector $\overline{r}$ with components defined by 
    $$\overline{r}_x=\min_{\{i\,|\,\phi_{ix}>0\}}J^*(i)+\frac{\epsilon}{1-\alpha},\qquad \hbox{for all }x\in{\cal A},$$
    so that 
    $$\overline{r}_x\leq J^*(i)+\frac{\epsilon}{1-\alpha},\qquad \hbox{for all $x$ and $i$ such that }\phi_{ix}>0.$$
    For all $x\in \mathcal{A}$, we have 
    \begin{align*}
        (H\overline{r})(x)=&\sum_{i=1}^nd_{xi}\min_{u\in U(i)}\sum_{j=1}^np_{ij}(u)\Bigg(g(i,u,j)+\alpha \sum_{y\in \mathcal{A}}\phi_{jy}\overline{r}_y\Bigg)\\
        =&\sum_{\{i\,|\,d_{xi}>0\}}d_{xi}\min_{u\in U(i)}\sum_{j=1}^np_{ij}(u)\Bigg(g(i,u,j)+\alpha \sum_{y\in \mathcal{A}}\phi_{jy}\overline{r}_y\Bigg)\\
        =&\sum_{\{i\,|\,d_{xi}>0\}}d_{xi}\min_{u\in U(i)}\sum_{j=1}^np_{ij}(u)\Bigg(g(i,u,j)+\alpha \sum_{\{y\,|\,\phi_{jy}>0\}}\phi_{jy}\overline{r}_y\Bigg)\\
        \leq &\sum_{\{i\,|\,d_{xi}>0\}}d_{xi}\min_{u\in U(i)}\sum_{j=1}^np_{ij}(u)\Bigg(g(i,u,j)+\alpha \sum_{\{y\,|\,\phi_{jy}>0\}}\phi_{jy}\bigg(J^*(j)+\frac{\epsilon}{1-\alpha}\bigg)\Bigg)\\
        = &\sum_{\{i\,|\,d_{xi}>0\}}d_{xi}\min_{u\in U(i)}\sum_{j=1}^np_{ij}(u)\Bigg(g(i,u,j)+\alpha J^*(j)+\frac{\alpha\epsilon}{1-\alpha}\Bigg)\\
        =&\sum_{{\{i\,|\,d_{xi}>0\}}}d_{xi}J^*(i)+\frac{\alpha\epsilon}{1-\alpha}\\
        \leq& \sum_{\{i\,|\,d_{xi}>0\}}d_{xi}\bigg(\min_{\{i'\,|\,\phi_{i'x}>0\}} J^*(i')+\epsilon\bigg)+\frac{\alpha\epsilon}{1-\alpha}\\
        =&\ \overline{r}_x,
    \end{align*}
    where the last inequality holds since for all $i$ with $d_{xi}>0$, we have $\phi_{ix}>0$ [by condition \eqref{eq:critical_condition}], so that
    $$J^*(i)\le \min_{\{i'\,|\,\phi_{i'x}>0\}} J^*(i')+\epsilon$$ 
    (by the definition of $\epsilon$). 
    
    Thus, we have $H\overline{r}\leq \overline{r}$, which implies that $r^*\leq \overline{r}$ [this is obtained by repeatedly iterating with $H$ both sides of the inequality $H\overline{r}\leq \overline{r}$, and by using the monotonicity of $H$, cf.\ Eq.\ \eqref{eq:monotone}, and the contraction property of $H$]. In view of the definition of $\overline{r}$, the inequality $r^*\leq \overline{r}$ is written as
    $$r^*_x\leq\min_{\{i\,|\,\phi_{ix}>0\}}J^*(i)+\frac{\epsilon}{1-\alpha},\qquad \hbox{for all }x\in \mathcal{A},$$
    which implies that
    $$r_x^*\leq J^*(i)+\frac{\epsilon}{1-\alpha},\qquad \hbox{for all $x$ and $i$ such that }\phi_{ix}>0.$$
    As a result, we obtain
     \begin{equation}
     \tilde J(i)=\sum_{x\in \mathcal{A}}\phi_{ix}r^*_x=\sum_{\{x\,|\,\phi_{ix}>0\}}\phi_{ix}r^*_x\leq J^*(i)+\frac{\epsilon}{1-\alpha},\qquad \hbox{for all }i=1,\ldots,n.
    \label{rightside}
\end{equation}

    The reverse direction can be shown by considering the vector $\underline{r}$ with components defined by 
    $$\underline{r}_x=\max_{\{i\,|\,\phi_{ix}>0\}}J^*(i)-\frac{\epsilon}{1-\alpha},\qquad \hbox{for all }x\in{\cal A}.$$
    Similarly, we have
    $$\underline{r}_x\geq J^*(i)-\frac{\epsilon}{1-\alpha},\qquad \hbox{for all $x$ and $i$ such that }\phi_{ix}>0.$$
    Thus,
    \begin{align*}
        (H\underline{r})(x)=&\sum_{i=1}^nd_{xi}\min_{u\in U(i)}\sum_{j=1}^np_{ij}(u)\Bigg(g(i,u,j)+\alpha \sum_{y\in \mathcal{A}}\phi_{jy}\underline{r}_y\Bigg)\\
        =&\sum_{\{i\,|\,d_{xi}>0\}}d_{xi}\min_{u\in U(i)}\sum_{j=1}^np_{ij}(u)\Bigg(g(i,u,j)+\alpha \sum_{y\in \mathcal{A}}\phi_{jy}\underline{r}_y\Bigg)\\
        =&\sum_{\{i\,|\,d_{xi}>0\}}d_{xi}\min_{u\in U(i)}\sum_{j=1}^np_{ij}(u)\Bigg(g(i,u,j)+\alpha \sum_{\{y\,|\,\phi_{jy}>0\}}\phi_{jy}\underline{r}_y\Bigg)\\
        \geq &\sum_{\{i\,|\,d_{xi}>0\}}d_{xi}\min_{u\in U(i)}\sum_{j=1}^np_{ij}(u)\Bigg(g(i,u,j)+\alpha \sum_{\{y\,|\,\phi_{jy}>0\}}\phi_{jy}\bigg(J^*(j)-\frac{\epsilon}{1-\alpha}\bigg)\Bigg)\\
        = &\sum_{\{i\,|\,d_{xi}>0\}}d_{xi}\min_{u\in U(i)}\sum_{j=1}^np_{ij}(u)\Bigg(g(i,u,j)+\alpha J^*(j)-\frac{\alpha\epsilon}{1-\alpha}\Bigg)\\
        =&\sum_{{\{i\,|\,d_{xi}>0\}}}d_{xi}J^*(i)-\frac{\alpha\epsilon}{1-\alpha}\\
        \geq& \sum_{\{i\,|\,d_{xi}>0\}}d_{xi}\bigg(\max_{\{i'\,|\,\phi_{i'x}>0\}} J^*(i')-\epsilon\bigg)-\frac{\alpha\epsilon}{1-\alpha}\\
        =&\ \underline{r}_x.
    \end{align*}
    Hence, we have $H\underline{r}\geq\underline{r}$, which implies $r^*\geq \underline{r}$. In view of the definition of $\underline{r}$, we have
    $$r^*_x\geq\max_{\{i\,|\,\phi_{ix}>0\}}J^*(i)-\frac{\epsilon}{1-\alpha},\qquad \hbox{for all }x\in \mathcal{A},$$
    which implies that
    $$r_x^*\geq J^*(i)-\frac{\epsilon}{1-\alpha},\qquad \hbox{for all $x$ and $i$ such that }\phi_{ix}>0.$$
    As a result, we obtain
   \begin{equation}
   \tilde J(i)=\sum_{x\in \mathcal{A}}\phi_{ix}r^*_x=\sum_{\{x\,|\,\phi_{ix}>0\}}\phi_{ix}r^*_x\geq J^*(i)-\frac{\epsilon}{1-\alpha},\qquad \hbox{for all }i=1,\ldots,n.
        \label{leftside}
\end{equation}
The desired bound now follows from Eqs.~\eqref{rightside} and \eqref{leftside}.
\end{proof}

We provide an example, which shows that the condition \eqref{eq:critical_condition} is essential for our bound to hold.

\begin{example}
[A Counterexample] 
\label{examplefinitehorizon}

\begin{figure}
    \centering
    \includegraphics[width=0.75\linewidth]{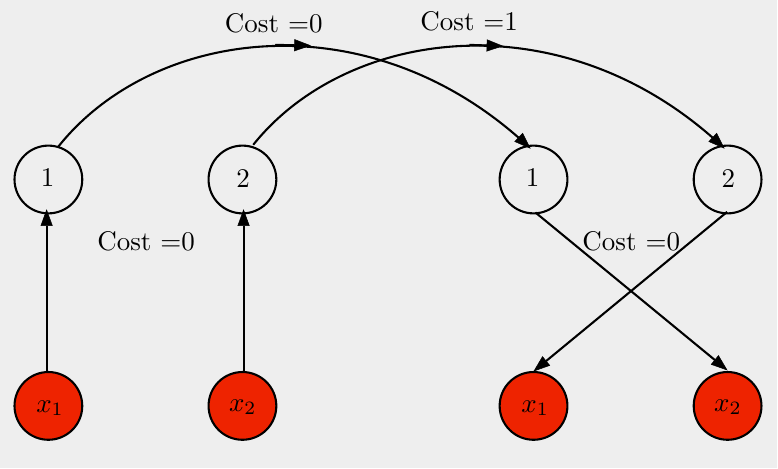}
    \caption{Illustration of the aggregate problem, and the corresponding transition mechanism and costs per stage for the problem of Example~\ref{examplefinitehorizon}. All the transitions shown have probability equal to 1.}
    \label{fig:counterexample}
\end{figure}

Consider a system involving two absorbing states, 1 and 2, i.e.,
$$p_{11}=1,\qquad p_{22}=1,\qquad p_{12}=0,\qquad p_{21}=0,$$
with self transition costs 
$$g(1,1)=0,\qquad g(2,2)=1.$$
Thus the infinite horizon costs (without aggregation) are
$$J^*(1)=0,\qquad J^*(2)=\frac{1}{1-\alpha}.$$
Assume that there are two aggregate states $x_1$ and $x_2$ that disaggregate into states 1 and 2, respectively, but aggregate states 2 and 1, respectively, i.e., 
$$d_{x_11}=1,\qquad  d_{x_22}=1,\qquad  \phi_{1x_2}=1,\qquad  \phi_{2x_1}=1;$$
see Fig.~\ref{fig:counterexample}. Then it can be seen that $\epsilon=0$ (since the sets $\{j\,|\,\phi_{jx_1}>0\}$ and $\{j\,|\,\phi_{jx_2}>0\}$ consist of a single state), but the true aggregation error is positive, i.e., the aggregation process is not exact. Indeed the sequence of generated costs starting from aggregate state $x_1$ is
$$\{0,\alpha,0,\alpha^3,0,\alpha^5,\ldots\},$$
while the sequence of generated costs starting from aggregate state $x_2$ is
$$\{1,0,\alpha^2,0,\alpha^4,0,\alpha^6,\ldots\},$$
so we have $\tilde J(i)\neq J^*(i)$ for both states $i=1,2$.

It is interesting to note that if we change the aggregation probabilities to
$$\phi_{1x_1}=\delta,\qquad  \phi_{2x_2}=\delta,\qquad  \phi_{1x_2}=1-\delta,\qquad  \phi_{2x_1}=1-\delta,$$
our bound holds for all $\delta\in(0,1]$, since the condition \eqref{eq:critical_condition} is satisfied. The bound fails to hold in the limit where $\delta=0$.
\end{example}

We note that our bound is conservative and can be quite poor, since the scalar $\epsilon$ depends only on the sets $\{j\,|\, \phi_{jx}>0\}$, $x\in\cal{A}$, and not on the actual values of $\phi_{jx}$. This can also be verified with the preceding example for $\delta>0$. An interesting question is whether the bound can be improved by proper selection of the aggregate states, based on some a priori knowledge of a good approximation to the optimal cost function $J^*$. Ideas of adaptive aggregation (Bertsekas and Castanon \cite{bertsekas1989adaptive}), and biased aggregation (Bertsekas \cite{bertsekas2019reinforcement}, Section 6.5, \cite{bertsekas2019biased}) may prove useful in this regard.

\bibliographystyle{alpha}
\bibliography{ref}
\end{document}